\documentclass[12pt]{amsart}
\usepackage{SSdefn}
\usepackage[lite,nobysame,alphabetic,backrefs]{amsrefs}
\usepackage{verbatim,extarrows}

\newcommand{\invl}{\llbracket\hspace{-.12cm} \llbracket}
\newcommand{\invr}{\rrbracket\hspace{-.12cm} \rrbracket}
\newcommand{\vars}{\invl x_1,x_2,\dots\invr}

\newtheorem*{theorem*}{Theorem}
\newtheorem{hyp}[equation]{Hypothesis}
\newenvironment{hypothesis}[1][]{\begin{hyp}[#1] \pushQED{\qed}}{\popQED \end{hyp}}

\newcommand{\stacks}[1]{\cite[Tag \href{https://stacks.math.columbia.edu/tag/#1}{#1}]{stacks-project}}

\title{Big polynomial rings and Stillman's Conjecture}

\author{Daniel Erman}
\address{Department of Mathematics, University of Wisconsin, Madison, WI}
\email{\href{mailto:derman@math.wisc.edu}{derman@math.wisc.edu}}
\urladdr{\url{http://math.wisc.edu/~derman/}}

\author{Steven V Sam}
\address{Department of Mathematics, University of Wisconsin, Madison, WI}
\curraddr{Department of Mathematics, University of California, San Diego, CA}
\email{\href{mailto:ssam@ucsd.edu}{ssam@ucsd.edu}}
\urladdr{\url{http://math.ucsd.edu/~ssam/}}

\author{Andrew Snowden}
\address{Department of Mathematics, University of Michigan, Ann Arbor, MI}
\email{\href{mailto:asnowden@umich.edu}{asnowden@umich.edu}}
\urladdr{\url{http://www-personal.umich.edu/~asnowden/}}

\thanks{DE was partially supported by NSF DMS-1302057 and NSF DMS-1601619.}
\thanks{SS was partially supported by NSF DMS-1500069 and DMS-1651327 and a Sloan Fellowship.}
\thanks{AS was supported by NSF DMS-1303082 and DMS-1453893 and a Sloan Fellowship.}

\date{July 4, 2022}


\begin{document}

\maketitle

\section{Introduction}

The purpose of this paper is to prove that certain limits of polynomial rings are themselves polynomial rings, and show how this observation can be used to deduce some interesting results in commutative algebra. In particular, we give two new proofs of Stillman's conjecture \cite[Problem 3.14]{stillman}. The first is similar to that of Ananyan--Hochster \cite{ananyan-hochster}, though more streamlined; in particular, it establishes the existence of small subalgebras. The second proof is completely different, and relies on a recent noetherianity result of Draisma \cite{draisma}.

\subsection{Polynomiality results}

For a commutative ring $A$, let $A \vars$ be the inverse limit of the standard-graded polynomial rings $A[x_1,\ldots,x_n]$ in the category of graded rings. A degree $d$ element of this ring is a (possibly infinite) formal $A$-linear combination of degree $d$ monomials in the variables $\{x_i\}_{i \ge 1}$. Fix a field $\bk$, and let $\bR=\bk\vars$. Our first polynomiality theorem is:

\begin{theorem} \label{introthm1}
Assume $\bk$ is perfect. Then $\bR$ is (isomorphic to) a polynomial ring. 
\end{theorem}

The set of variables in the polynomial ring is uncountable; hence the phrase ``big polynomial rings'' in the title of the paper. We deduce Theorem~\ref{introthm1} from the following general criterion. For a graded ring\footnote{In this paper, all graded rings are supported in non-negative degrees.} $R$, we write $R_+$ for the ideal of positive degree elements, and we write $R_+^2$ for the square of this ideal.

\begin{theorem} \label{polycrit}
Let $R$ be a graded ring with $R_0=\bk$ a perfect field. Assume:
\begin{itemize}
\item Characteristic zero: $R$ has enough derivations (Definition~\ref{defn:enough1}), that is, for every non-zero $x \in R_+$ there is a derivation $\partial$ of negative degree such that $\partial(x) \ne 0$.
\item Positive characteristic: $R$ has enough Hasse derivations (see Definition~\ref{defn:enough2}).
\end{itemize}
Then $R$ is isomorphic to a polynomial ring.  Precisely, for any set of positive degree homogeneous elements $\{f_i\}_{i \in I}$ whose images in $R_+/R_+^2$ form a $\bk$-basis, the $\bk$-algebra homomorphism $\bk[X_i]_{i \in I} \to R$ taking $X_i$ to $f_i$ is an isomorphism.
\end{theorem}

The proof of Theorem~\ref{polycrit} is elementary. Surjectivity follows immediately by induction on degree: any homogeneous element in $R_+^2$ is, by definition, generated by elements of lower degree. For injectivity, if one had an algebraic relation among some of the $f_i$, then one could apply an appropriate (Hasse) derivation to get a lower degree relation, and eventually reach a contradiction. To then obtain Theorem~\ref{introthm1} from Theorem~\ref{polycrit}, we simply observe that (Hasse) derivatives with respect to the variables $x_i$ extend continuously to $\bR$ and furnish it with enough (Hasse) derivations.

The inverse limit $\bR$ is one way to make sense of a limit of finite polynomial rings. A different way is through the use of ultrapowers, or, more generally, ultraproducts (see \S \ref{subsec:ultra} for background). Let $\bS$ be the graded ultrapower of the standard-graded polynomial ring $\bk[x_1,x_2,\ldots]$. We also prove:

\begin{theorem} \label{introthm2}
Assume $\bk$ is perfect. Then $\bS$ is a polynomial ring.
\end{theorem}

This also follows quickly from Theorem~\ref{polycrit}. The perfectness hypotheses in this section can be relaxed: for instance, Theorems~\ref{introthm1} and~\ref{introthm2} hold if $[\bk:\bk^p]$ is finite, see Remarks~\ref{rmk:relax} and~\ref{rmk:relax2}. In fact, some time after completing this paper, we succeeded in removing these hypotheses entirely; see \cite{imperfect}.

\subsection{Connection to the work of Ananyan--Hochster}

We recall (and slightly extend) the notion of strength from \cite{ananyan-hochster}:

\begin{definition}
Let $R$ be a graded ring with $R_0=\bk$ a field, and let $f$ be a homogeneous element of $R$. The {\bf strength} of $f$ is the minimal integer $k \ge -1$ for which there is a decomposition $f=\sum_{i=1}^{k+1} g_i h_i$ with $g_i$ and $h_i$ homogeneous elements of $R$ of positive degree, or $\infty$ if no such decomposition exists. The {\bf collective strength} of a set of homogeneous elements $\{f_i\}_{i \in I}$ of $R$ is the minimal strength of a non-trivial homogeneous $\bk$-linear combination.
\end{definition}

\begin{example} \label{ex:strength}
\hangindent\leftmargini
\textup{(a)}\hskip\labelsep In $\bk[x_1, \ldots, x_n]$, non-zero elements of degree $1$ have infinite strength, while non-zero elements of degree $>1$ have strength $<n$.
\begin{enumerate}[topsep=0pt]
\setcounter{enumi}{1}
\item In $\bR$, there are a wealth of interesting elements of infinite strength, such as $\sum_{i \ge 1} x_i^d$ (if $d$ is invertible in $\bk$).\footnote{To see this, let $F=\sum_i x_i^d$.  If $F$ decomposes as $F=f_0g_0+ \cdots +f_sg_s$ then the singular locus of $V(F)$ would contain $V(f_0,\dots,f_s,g_0,\dots,g_s)$ and hence would have codimension at most $2s+2$.  However, if $d$ is invertible, then the singular locus of $V(F)$ has infinite codimension.}
\item In any graded ring $R$, the ideal $R_+^2$ is exactly the ideal of finite strength elements. \qedhere
\end{enumerate}
\end{example}

Many of the results of Ananyan--Hochster are instances of the following general principle: elements in a polynomial ring of sufficiently large collective strength behave approximately like independent variables. Theorem~\ref{introthm1} shows that this approximation becomes exact in the limiting ring $\bR$. Indeed, suppose $\{f_i\}_{i \in I}$ are elements of $\bR_+$ that form a basis modulo $\bR_+^2$. Thus no linear combination of the $f_i$ belongs to $\bR_+^2$, i.e., has finite strength (Example~\ref{ex:strength}(c)), and so $\{f_i\}$ has infinite collective strength. The Ananyan--Hochster principle thus suggests that the $\{f_i\}$ should be independent variables, and this is exactly the content of Theorem~\ref{introthm1}. 

\subsection{Stillman's conjecture via ultraproducts} \label{ss:stillmanS}
While ultraproducts may be less familiar to some readers than inverse limits, Theorem~\ref{introthm2} leads to our most efficient proof of Stillman's conjecture.
As in \cite{ananyan-hochster} (see \S \ref{ss:stillman}), both the existence of small subalgebras and Stillman's conjecture can be reduced to the following statement:

\begin{theorem}\label{thm:reg high strength}
Fix integers $d_1, \ldots, d_r$. Then there exists an integer $N$ with the following property. If $\bk$ is a perfect field, and $f_1, \ldots, f_r\in \bk[x_1,\dots,x_n]$ are polynomials of degrees $d_1, \ldots, d_r$ with collective strength at least $N$, then $f_1, \ldots, f_r$ is a regular sequence.
\end{theorem}

Ananyan--Hochster prove this theorem via a multi-tiered induction, where elements of increasingly high strength obtain an array of increasingly nice properties. Our proof using Theorem~\ref{introthm2} is more direct. Here is the idea. Suppose that $f_{1,i}, \ldots, f_{r,i} \in \bk[x_1,x_2,\dots]$, for $i \in \bN$, are polynomials of the given degrees with collective strength tending to infinity. It suffices to show that $f_{1,i}, \ldots, f_{r,i}$ eventually form a regular sequence. For each $j$, the sequence $f_{j,\bullet}$ defines an element $f_j$ in the ultraproduct ring $\bS$. It is easy to see that $f_1, \ldots, f_r$ have infinite collective strength (Proposition~\ref{prop:str}). Thus, by Theorem~\ref{introthm2}, $f_1, \ldots, f_r$ are independent variables in $\bS$, and hence form a regular sequence. We then apply a result (Corollary~\ref{cor:reg seq ultra}) comparing codimension in $\bS$ to codimension in $\bk[x_1,x_2,\ldots]$ to conclude that $f_{1,i}, \ldots, f_{r,i}$ is eventually a regular sequence.

As in \cite{ananyan-hochster}, we show that the bound in Theorem~\ref{thm:reg high strength} (and Stillman's conjecture as well) is independent of the field $\bk$. To do so, we prove a generalization of Theorem~\ref{introthm2} (see \S\ref{ss:ultramain}) where $\bS$ is replaced with an ultraproduct of polynomial rings with variable coefficient fields.

\subsection{Stillman's conjecture via inverse limits} \label{ss:stillmanR}
Returning to the inverse limit, Theorem~\ref{introthm1} enables a proof of Stillman's conjecture that follows the very general rubric in algebraic geometry of proving a result generically, spreading out to an open set, and then inductively treating proper closed subsets.  The basic idea in characteristic zero is as follows. Suppose that $A$ is a characteristic~0 domain with fraction field $K$, and $M$ is a finitely presented $A \vars$-module. Then $K \otimes_A M$ is a finitely presented module over the ring $K \otimes_A A\vars$.  While $K\otimes_A A\vars$ is generally not isomorphic to $K\vars$, Theorem~\ref{polycrit} shows it is also an abstract polynomial ring.

It then follows from simple homological properties of infinite polynomial rings that $K \otimes_A M$ has a finite length resolution by finite free modules. A flatness argument produces an open dense subset $U$ of $\Spec(A)$ such that $M_y$ has the same Betti table as $K \otimes_A M$ for all $y \in U$. We can then restrict our attention to $\Spec(A) \setminus U$, and apply the same argument. This shows that there is some (perhaps infinite) stratification of $\Spec(A)$ such that on each stratum the fibers of $M$ have the same Betti table.

We apply this as follows. Fix positive integers $d_1, \ldots, d_r$, and let $A$ be the symmetric algebra on the vector space $\Sym^{d_1}(\bk^{\infty}) \oplus \cdots \oplus \Sym^{d_r}(\bk^{\infty})$. Then $\Spec(A)$ is the space of forms $f_1, \ldots, f_r \in \bk\vars$ of degrees $d_1, \ldots, d_r$. We let $M$ be the universal module $A\vars/(f_1,\ldots,f_r)$.  The stratification constructed in the previous paragraph can be made compatible with the $\GL_{\infty}$ action on $\Spec(A)$. A recent theorem of Draisma \cite{draisma} asserts that $\Spec(A)$ is $\GL_{\infty}$-noetherian, and hence this stratification is finite. We conclude that there are only finitely many resolution types for ideals generated by $f_1, \ldots, f_r$ of the given degrees. This, in particular, implies Stillman's conjecture in characteristic zero.

The same idea works in positive characteristic, but when $K$ fails to be perfect, we need to bootstrap from the perfect case to produce the open subset with constant Betti numbers.

\subsection{Connections to other work}
The Milnor--Moore theorem \cite{MM}, and generalizations \cite{sjodin}, establish that certain commutative graded rings are polynomial rings via properties of a comultiplication. While this, and its extensions to non-commutative rings, can be applied to examples in commutative algebra, it is of a fairly distinct nature from the criteria in the present paper.

Theorem~\ref{introthm1} is an example of the meta-principle that inverse limits of free objects tend to be free themselves. See \cite[\S I.4.2, Corollary 4]{serre} for an example of this principle with pro-$p$-groups. Alexandru Chirvasitu informed us that he can prove a non-commutative version of Theorem~\ref{introthm1} where polynomial rings are replaced by non-commutative polynomial rings.

The use of ultraproducts in commutative algebra was famously employed in \cite{van-den-Dries-schmidt} to establish a variety of bounds (with the number of variables fixed). See \cite{schoutens} for more discussion and examples.

The Gr\"obner theory of the inverse limit ring $\bk\vars$ was studied by Snellman in~\cite{snellman,snellman-article}. Shortly after a draft of this article was posted, ~\cite{draisma-lason-leykin} applied Theorem~\ref{introthm1} to obtain finiteness results for grevlex Gr\"obner bases over $\bR$, and then used this to answer some questions raised by Snellman and to give a generic initial ideal proof of Stillman's Conjecture.

The use of $\GL_\infty$-noetherianity of spaces to prove the existence of uniform bounds in algebraic geometry has been used in several papers.  See \cite{draisma-survey} for a survey.

\subsection{Outline}

In \S \ref{s:poly}, we establish our polynomiality criteria (summarized in Theorem~\ref{polycrit}). In \S \ref{s:codim}, we prove some easy results concerning dimension theory in polynomial rings with an infinite number of variables. In \S \ref{s:ultra}, we prove that the ultraproduct ring is a polynomial ring (Theorem~\ref{introthm2}), and use this to deduce our first proof of Stillman's conjecture. Finally, in \S \ref{s:limit}, we prove that the inverse limit ring is a polynomial ring (Theorem~\ref{introthm1}), and use this to deduce our second proof of Stillman's conjecture.

\subsection*{Acknowledgements}
We thank Craig Huneke and Gregory G. Smith for useful conversations. We also thank Alexandru Chirvasitu for informing us about his work on the non-commutative analogue of Theorem~\ref{introthm1} and the reference in \cite{serre}.

\section{Criteria for polynomiality} \label{s:poly}

Let $R$ be a graded ring with $R_0=\bk$ a field. We say that $R$ {\bf is a polynomial ring} if there are elements $\{x_i\}_{i \in I}$ of $R$, each homogeneous of positive degree, such that the natural map $\bk[X_i]\to R$ sending $X_i$ to $x_i$ is an isomorphism. The $x_i$'s need not have degree~1, and the set $I$ need not be finite.  The purpose of this section is to characterize polynomial rings via derivations.

\subsection{Characteristic~0}

We first treat the case where $\bk$ has characteristic~0, for which the following definition and theorem constitute our criterion for polynomiality. We say that a derivation $\partial$ of a graded ring $R$ is {\bf homogeneous of degree $d$} if $\deg \partial(x)=\deg(x)+d$ for all homogeneous $x \in R$.

\begin{definition} \label{defn:enough1}
Let $R$ be a graded ring with $R_0=\bk$ a field. We say that $R$ {\bf has enough derivations} if for every non-zero homogeneous element $x$ of positive degree there is a homogeneous derivation $\partial$ of negative degree such that $\partial(x) \ne 0$.
\end{definition}

\begin{theorem} \label{thm:poly}
Let $R$ be a graded $\bk$-algebra with $R_0=\bk$ a field of characteristic~$0$. Then $R$ is a polynomial ring if and only if $R$ has enough derivations.
\end{theorem}

\begin{proof}
In this proof, ``derivation'' will mean ``homogeneous derivation of negative degree.'' It is clear that a polynomial ring has enough derivations. We prove the converse.

Let $\cE$ be a set of homogeneous elements of $R_+$ that gives a basis of $R_+/R_+^2$. By graded Nakayama's lemma, $\cE$ generates $R$ as a $\bk$-algebra, so it suffices to show that $\cE$ is algebraically independent. Let $\cE_{\le d}$ (resp.\ $\cE_d$) be the set of elements in $\cE$ of degree $\le d$ (resp.\ $d$). We prove that $\cE_{\le d}$ is algebraically independent for all $d$ by induction on $d$. Suppose that we have shown $\cE_{\le d-1}$ is algebraically independent. To prove that $\cE_{\le d}$ is algebraically independent, it suffices to prove the following statement: if $\cE_{\le d-1} \subset E \subset \cE_{\le d}$ is algebraically independent and $x \in \cE_d \setminus E$, then $E'=E \cup \{x\}$ is algebraically independent. Indeed, this statement implies that all sets of the form $\cE_{\le d-1} \cup E''$ with $E''$ a finite subset of $\cE_d$ are algebraically independent, which implies that $\cE_{\le d}$ is algebraically independent.

Thus let $E$, $E'$, and $x$ as above be given. Let $A \subset R$ be the $\bk$-subalgebra generated by $E$. To prove that $E'$ is algebraically independent, it suffices to show that if $0=\sum_{i=0}^n a_i x^i$ with $a_i \in A$ then $a_i=0$ for all $i$. Before proceeding, we note that if $\partial$ is any derivation of $R$ then $\partial(\cE_{\le d}) \subset A$ since $\partial$ decreases degrees, and so $\partial(A) \subset A$ and $\partial(x) \in A$.

Suppose that $0=\sum_{i=0}^n a_i x^i$ with $a_i \in A$ and $a_n \ne 0$. Of all such relations, choose a homogeneous one of minimal degree (i.e., with $\deg(a_nx^n)$ minimal). Suppose that $a_n$ has positive degree. By assumption, there exists a derivation $\partial$ such that $\partial(a_n) \ne 0$. Applying $\partial$ to our given relation yields $0=\partial(a_n) x^n + \sum_{i=0}^{n-1} b_i x^i$ where the $b_i$ are elements of $A$. This is a contradiction, since $\partial(a_n)$ has smaller degree than $a_n$. Thus $\deg(a_n)=0$, and so we may assume $a_n=1$.

Since $\cE$ is linearly independent modulo $R_+^2$, we see that $x\notin A$, and so $n \ge 2$ and $nx+a_{n-1}$ is non-zero.  It follows that there exists a derivation $\partial$ such that $\partial(nx+a_{n-1}) \ne 0$. Applying $\partial$ to our original relation gives $0=\partial(nx+a_{n-1}) x^{n-1} + \sum_{i=0}^{n-2} b_i x^i$ for some $b_i \in A$. This is a smaller degree relation, which is a contradiction. We thus see that no relation $0=\sum_{i=0}^n a_i x^i$ exists with $a_n$ non-zero, which completes the proof.
\end{proof}

\subsection{Positive characteristic}

Theorem~\ref{thm:poly} obviously fails in characteristic $p$: since $p$th powers are killed by every derivation, no reduced ring has enough derivations. The most obvious adjustment would be to ask that if $x$ is a homogeneous element of $R$ that is not a $p$th power then there is a derivation $\partial$ such that $\partial(x) \ne 0$. The following two examples show that this condition is insufficient to conclude that $R$ is a polynomial ring.

\begin{example}
Let $R=\bk[x]/(x^p)$ where $\bk$ is perfect of characteristic~$p$ and $x$ has degree~1. Then $\frac{d}{dx}$ is a well-defined derivation on $R$, and thus $R$ has enough derivations.
\end{example}

\begin{example}
Let $R=\bk[x,y,\tfrac{y}{x^p}]$ where $\bk$ is perfect of characteristic $p$, $x$ has degree~1, and $y$ has degree $p+1$. Then $\frac{\partial}{\partial x}$ and $x^p \frac{\partial}{\partial y}$ are well-defined derivations on $R$, and every homogeneous element of $R$ that is not a $p$th power is not annihilated by one of them.
\end{example}

To extend our criterion to the positive characteristic case, we employ the following extension of the notion of a derivation (see \cite[pp.\ 27--29]{goldschmidt} for additional discussion).

\begin{definition}
Let $R$ be a $\bk$-algebra. A {\bf Hasse derivation} on $R$ is a sequence $\partial^{\bullet}=(\partial^n)_{n \ge 0}$ where each $\partial^n$ is a $\bk$-linear endomorphism of $R$ such that $\partial^0$ is the identity and
\begin{displaymath}
\partial^n(xy)=\sum_{i+j=n} \partial^i(x) \partial^j(y)
\end{displaymath}
holds for all $x,y \in R$. If $R$ is graded then we say $\partial^{\bullet}$ is {\bf homogeneous of degree $d$} if $\partial^n(x)$ has degree $\deg(x)+nd$ for all homogeneous $x \in R$.
\end{definition}

\begin{remark}
Giving a Hasse derivation on $R$ is equivalent to giving a ring homomorphism $\varphi \colon R \to R \lbb t \rbb$ such that the constant term of $\varphi(x)$ is $x$. If $\partial^{\bullet}$ is a Hasse derivation, then the associated ring homomorphism is defined by $\varphi(x) = \sum_{i \ge 0} \partial^i(x) t^i$.
\end{remark}

\begin{example} \label{ex:hasse}
Suppose $R=\bk[x]$, with $\bk$ any field. Define $\partial^n(x^k) = \binom{k}{n} x^{k-n}$. (Note that $\partial^n = \frac{1}{n!} \frac{d^n}{dx^n}$ if $n!$ is invertible in $\bk$.)  Then $\partial^{\bullet}$ is a Hasse derivation, called the {\bf Hasse derivative}. If $R$ is graded with $x$ of degree $d$ then $\partial^{\bullet}$ is homogeneous of degree $-d$. The homomorphism $\varphi \colon R \to R\lbb t \rbb$ associated to the Hasse derivative is given by $x \mapsto x+t$.
\end{example}

\begin{remark}\label{rmk:draisma hasse}
Curiously, Hasse derivatives also play a key role in Draisma's~\cite{draisma}, where they are closely related to his directional derivatives.
\end{remark}

\begin{lemma} \label{lem:der-powers}
Let $R$ be a $\bk$-algebra, where $\bk$ is a field of characteristic~$p$, and let $\partial^{\bullet}$ be a Hasse derivation on $R$. Let $q$ be a power of $p$. Then for $x \in R$ and $n \in \bN$ we have
\begin{displaymath}
\partial^n(x^q) = \begin{cases}
(\partial^{n/q}{x})^q & \text{if $q \mid n$} \\
0 & \text{if $q \nmid n$} \end{cases}.
\end{displaymath}
\end{lemma}

\begin{proof}
We have
\begin{displaymath}
\partial^n(x^q) = \sum_{\substack{(i_1,\dots,i_q)\\i_1+\cdots+i_q=n}} \partial^{i_1}(x) \cdots \partial^{i_q}(x).
\end{displaymath}
If $i_1, \ldots, i_q$ are not all equal then the orbit of $(i_1, \ldots, i_q)$ under the symmetric group $S_q$ has cardinality divisible by $p$. All elements of this orbit contribute equally to the sum, and thus they all cancel. We thus see that the only surviving term occurs when $n$ is a multiple of $q$ and $i_1=\cdots=i_q=n/q$; this term is $(\partial^{n/q}{x})^q$.
\end{proof}

The following definition and theorem constitute our criterion for polynomiality in positive characteristic.

\begin{definition} \label{defn:enough2}
Let $R$ be a graded ring with $R_0=\bk$ a field of characteristic~$p>0$. Let $R^p = \{f^p \mid f \in R\}$ denote the subring of $p$th powers in $R$. We say that $R$ {\bf has enough Hasse derivations} if the following condition holds: if $x$ is a positive degree homogeneous element of $R$ such that $x \not\in \bk R^p$ (the $\bk$-span of the set $R^p$) then there exists a homogeneous Hasse derivation $\partial^{\bullet}$ of $R$ of negative degree such that $\partial^1(x) \ne 0$.
\end{definition}

\begin{theorem} \label{thm:poly2}
Let $R$ be a graded ring with $R_0=\bk$ a perfect field of characteristic~$p>0$. Then $R$ is a polynomial ring if and only if it has enough Hasse derivations.
\end{theorem}

\begin{proof}
In this proof, ``Hasse derivation'' will mean ``homogeneous Hasse derivation of negative degree.'' We note that since $\bk$ is perfect, $\bk R^p=R^p$. If $R$ is a polynomial ring then it has enough Hasse derivations; one can see this using Hasse derivatives (Example~\ref{ex:hasse}). We now prove the converse.

We first show that $R$ is reduced. Suppose not, and let $x \in R$ be a non-zero homogeneous nilpotent element of minimal degree. Note that $x \not\in R^p$, for if $x=y^p$ then $y$ would be a lower degree nilpotent element. Let $r$ be such that $x^{p^r}=0$ and let $\partial^{\bullet}$ be a Hasse derivation such that $\partial^1(x) \ne 0$. Then $0=\partial^{p^r}(x^{p^r})=(\partial^1{x})^{p^r}$ (Lemma~\ref{lem:der-powers}), and so $\partial^1(x)$ is nilpotent, contradicting the minimality of $x$. Thus $R$ is reduced.

Let $\cE$ be a set of homogeneous elements of $R_+$ that forms a basis for $R_+/R_+^2$. It suffices to prove that $\cE$ is algebraically independent. For $E \subset \cE$, consider the following statement:
\begin{itemize}[leftmargin=3.5em]
\item[$\sA_E$:] Given distinct elements $x_1, \ldots ,x_r \in E$ and a polynomial $F \in \bk[X_1, \ldots, X_r]$ such that $F(x_1, \ldots, x_r) \in R^p$, we have $F \in \bk[X_1, \ldots, X_r]^p$.
\end{itemize}
Observe that if $\sA_E$ holds then $E$ is algebraically independent. Indeed, suppose that $F(x_1, \ldots, x_r)=0$ is a minimal degree algebraic relation among distinct elements of $E$. Since $0 \in R^p$, we see that $F(X_1, \ldots, X_r)=G(X_1, \ldots, X_r)^p$ for some $G$ by $\sA_E$, and so $G(x_1, \ldots, x_r)^p=0$. Since $R$ is reduced, it follows that $G(x_1, \ldots, x_r)=0$, contradicting the minimality of $F$. Thus to prove the theorem it suffices to prove $\sA_{\cE}$.

We prove that $\sA_E$ holds for all $E$ by induction on $E$ in the following manner. Let $\cE_{\le d}$ (resp.\ $\cE_d$) be the set of elements of $\cE$ of degree $\le d$ (resp.\ $d$). Suppose that $\cE_{\le d-1} \subset E \subset \cE_{\le d}$ and let $E'=E \cup \{x\}$ for some $x \in \cE_d \setminus E$. Assuming $\sA_E$, we prove $\sA_{E'}$. This will establish $\sA_E$ for all $E$ by the same logic used in the proof of Theorem~\ref{thm:poly}.

Fix $E$, $E'$, and $x$ as above. Let $A$ be the $\bk$-subalgebra of $R$ generated by $E$. We claim that $\sA_{E'}$ can be reduced to the following statement, for all $n$ and $m$:
\begin{itemize}[leftmargin=4.3em]
\item[$\sB_{n,m}$:] If $\sum_{i=0}^n a_i x^i \in R^p$ with $a_i \in A$ and $\deg(a_n) \le m$, then $a_i \in R^p$ and $ia_i=0$ for all $i$.
\end{itemize}
Indeed, suppose $\sB_{n,m}$ holds for all $n$ and $m$, and suppose $F(x_1,\ldots,x_r) \in R^p$ for distinct elements $x_1, \ldots, x_r \in E'$. We may as well suppose $x_r=x$ and $x_1,\ldots,x_{r-1} \in E$. Write $F(X_1,\ldots,X_r)=\sum_{i=0}^n G_i(X_1,\ldots,X_{r-1}) X_r^i$ for polynomials $G_i$. By $\sB_{n,m}$, we see that $G_i(x_1,\ldots,x_{r-1}) \in R^p$ for all $i$ and $G_i(x_1,\ldots,x_{r-1})=0$ if $p \nmid i$. By $\sA_E$, it follows that $G_i(X_1,\ldots,X_{r-1})=G'_i(X_1,\ldots,X_{r-1})^p$ for some polynomial $G'_i$ and that $G_i(X_1,\ldots,X_r)=0$ if $p \nmid i$. We thus find
\begin{displaymath}
F(X_1,\ldots,X_r)=\bigg( \sum_{\substack{0 \le i \le n\\ p \mid i}} G_{i}'(X_1,\ldots,X_{r-1}) X_r^{i/p} \bigg)^p,
\end{displaymath}
which establishes $\sA_{E'}$.

We now prove $\sB_{n,m}$ by induction on $n$ and $m$. Clearly, $\sB_{0,m}$ holds for all $m$. We note that if $\sB_{n,m}$ holds and $\sum_{i=0}^n a_i x^i=0$ with $\deg(a_n) \le m$ then $a_i=0$ for all $i$; the proof is the same as the proof given above that $\sA_E$ implies algebraic independence of $E$. We also note that if $\partial^{\bullet}$ is any Hasse derivation then $\partial^n(\cE_{\le d}) \subset A$ for all $n>0$, and so $\partial^n(A) \subset A$ and $\partial^n(x) \in A$.

We now prove $\sB_{1,m}$ for all $m$ by induction on $m$. First suppose $m=0$, and suppose that $ax+b \in R^p$ with $a \in \bk$ and $b \in A$. Write $b = b_0 + b'$ where $b_0$ is the degree $0$ piece in the homogeneous decomposition of $b$. Then $b_0 \in \bk^p$ and $ax+b'=0$ in $R_+/R_+^2$. Since $\cE$ is linearly independent in $R_+/R_+^2$, it follows that $a=0$, and so $\sB_{1,0}$ holds. Now suppose $\sB_{1,m-1}$ holds, and let us prove $\sB_{1,m}$. Thus suppose that $ax+b=y^p$ for some $y \in R^p$ with $a,b \in A$ and $\deg(a)\leq m$. If $\partial^{\bullet}$ is any Hasse derivation of $R$ then $\partial^1(a)x+(a\partial^1(x)+\partial^1(b))=0$ (Lemma~\ref{lem:der-powers}). Since $\deg(\partial^1(a))<m$, we see that $\partial^1(a)=0$ by $\sB_{1,m-1}$. Since this holds for all $\partial^{\bullet}$, we find $a \in R^p$. Suppose $a \ne 0$, and let $q$ be the maximal power of $p$ such that $a \in R^q$ (this exists since $\deg(a)>0$). Write $a=c^q$, and note $c \not\in R^p$. Let $\partial^{\bullet}$ be a Hasse derivation such that $\partial^1(c) \ne 0$; note then that $\partial^q(a)=(\partial^1{c})^q \ne 0$ (Lemma~\ref{lem:der-powers}). Again by Lemma~\ref{lem:der-powers}, we have
\begin{displaymath}
\partial^q(a)x+(a\partial^q(x)+\partial^q(b))=\partial^q(y^p)=(\partial^{q/p}{y})^p \in R^p
\end{displaymath}
By $\sB_{1,m-1}$, we have $\partial^q(a)=0$, a contradiction. Thus $a=0$ and $\sB_{1,m}$ holds.

We now prove $\sB_{n,m}$ for $n \ge 2$, assuming $\sB_{k,m}$ for all $1\leq k\leq n-1$ and assuming $\sB_{n,m-1}$. Thus suppose that $\sum_{i=0}^n a_i x^i \in R^p$ with $a_i \in A$ and $\deg(a_n) \le m$. Let $\partial^{\bullet}$ be a Hasse derivation of $R$. Applying $\partial^1$, we find
\begin{displaymath}
0=\partial^1(a_n) x^n + (n a_n \partial^1(x)+\partial^1(a_{n-1})) x^{n-1} + \cdots,
\end{displaymath}
where the remaining terms have degree $\le n-2$ in $x$. By $\sB_{n,m-1}$, all the above coefficients vanish. Thus $\partial^1(a_n)=0$ for all $\partial^{\bullet}$, and so $a_n\in R^p$. We now see that the coefficient of $x^{n-1}$ is $\partial^1(n a_n x + a_{n-1})$. Since this vanishes for all $\partial^{\bullet}$, we find $na_nx+a_{n-1} \in R^p$, and so $na_n=0$ by $\sB_{1,m}$. In particular, $p \mid n$ if $a_n \ne 0$, so $a_n x^n \in R^p$, and hence $\sum_{i=0}^{n-1} a_i x^i \in R^p$. Thus by $\sB_{n-1,m}$ we have $a_i \in R^p$ and $i a_i=0$ for all $0 \le i \le n-1$. This proves $\sB_{n,m}$.
\end{proof}

\begin{remark} \label{rmk:relax}
The perfectness hypothesis in Theorem~\ref{thm:poly2} can be omitted. Indeed, letting $\bK$ be the perfection of $\bk$, the theorem shows that $\bK \otimes_{\bk} R$ is a polynomial ring, which implies that $R$ is a polynomial ring.
\end{remark}

\section{Dimension theory in polynomial rings} \label{s:codim}

Fix a field $\bk$. For a ring $A$ and a (possibly infinite) set $\cU$, we let $A[\cU]$ be the polynomial algebra over $A$ in variables $\cU$. We aim to prove a number of basic results on codimension in rings of the form $A[\cU]$ where $A$ is a finitely generated $\bk$-algebra. All of these results are standard when $\cU$ is finite. We do not impose any gradings in this section.

For a prime ideal $\fp$ in a commutative ring $R$, the {\bf codimension} (or {\bf height}) of $\fp$ is the maximum integer $c$ for which there exists a chain of prime ideals $\fp_0 \subsetneq \cdots \subsetneq \fp_c=\fp$, or $\infty$ if such chains exist with $c$ arbitrarily large. All ideals considered in this section are assumed to be non-unital. The {\bf codimension} of an arbitrary non-unital ideal $I$ of $R$ is the minimum of the codimensions of primes containing $I$, or $\infty$ if $I$ is not contained in any prime of finite codimension. This will be denoted $\codim_R(I)$. We start with a basic fact that we will cite often.

\begin{proposition} \label{prop:finflat}
Let $A \subset B$ be a flat integral extension of rings. For any ideal $I \subset B$, we have $\codim_B(I)=\codim_A(A \cap I)$.
\end{proposition}

\begin{proof}
We first prove the statement assuming that $I=\fp$ is prime.  Suppose that $\codim_B(\fp) \ge c$. Let $\fp_0 \subsetneq \fp_1\subsetneq \cdots \subsetneq \fp_c = \fp$ be a chain of distinct prime ideals. Let $\fq_i= A \cap \fp_i$. By incomparability \cite[Theorem~14.3(2)]{altman-kleiman}, the $\fq_i$ are distinct and thus $\codim_A(\fq_c) \geq c$. In particular, if $\codim_B(\fp)=\infty$, this shows that $\codim_A(A \cap \fp)=\infty$. Now suppose that $\codim_B(\fp)$ is finite and equal to $c$. If there were some longer chain of primes leading up to $\fq_c$, then by going down for flat extensions \cite[Theorem 14.11]{altman-kleiman}, we would have $\codim_B(\fp_c) > c$, which is a contradiction.  Thus $\codim_A(\fq_c)=c$ which finishes the special case when $I$ is prime.

Now consider the general case. Given a prime $\fp$ containing $I$, we have just shown that $\codim_A(A \cap \fp) = \codim_B(\fp)$. On the other hand, given a prime $\fq$ containing $A \cap I$, using \cite[Theorem 14.3(4)]{altman-kleiman}, there is a prime $\fp \supset I$ such that $A \cap \fp = \fq$. In particular, we deduce that $\codim_B(I) = \codim_A(A \cap I)$.
\end{proof}

\begin{proposition} \label{prop:fg}
Let $A$ be a finitely generated $\bk$-algebra and let $\fp$ be a prime ideal of $A[\cU]$ of finite codimension. Then $\fp$ is finitely generated.
\end{proposition}

\begin{proof}
Let $c$ be the codimension of $\fp$. We prove the statement by induction on $c$. First suppose that $c=0$. If $\fp=0$, then we are done. Otherwise, choose a nonzero element $g \in \fp$. Let $\cV \subset \cU$ be a finite subset such that $g$ belongs to $A[\cV]$. Let $\fp' = A[\cU](A[\cV] \cap \fp)$. Then we have $\fp' \subseteq \fp$. Since $\fp$ is prime, so is $A[\cV] \cap \fp$, and hence so is $\fp'$ since $A[\cU]$ is obtained from $A[\cV]$ by adjoining variables. In particular, we have $\fp'=\fp$, and so $\fp$ is finitely generated.

Now suppose $c>0$. Choose a prime ideal $\fq \subset \fp$ of codimension $c-1$. By induction, we know that $\fq$ is finitely generated; let $f_1, \ldots, f_r$ be generators. Let $g \in \fp \setminus \fq$ and let $\cV \subset \cU$ be a finite subset such that the $f_i$'s and $g$ belong to $A[\cV]$. Let $\fp' = A[\cU](A[\cV] \cap \fp)$. Note that $\fq = A[\cU](A[\cV] \cap \fq)$. We thus have $\fq \subset \fp' \subset \fp$. Since $\fp$ is prime, so is its contraction to $A[\cV]$, and so is the extension of this back to $A[\cU]$, since $A[\cU]$ is obtained from $A[\cV]$ by adjoining variables. Thus $\fp'$ is either $\fq$ or $\fp$; however, it is not $\fq$ since it contains $g$. Thus $\fp=\fp'$, which shows that $\fp$ is finitely generated.
\end{proof}

\begin{proposition} \label{prop:codimex}
Let $A$ be a finitely generated $\bk$-algebra and let $\cV \subset \cU$ be sets.  If $I$ is a finitely generated ideal of $A[\cV]$, and $J$ is its extension to $A[\cU]$, then $\codim_{A[\cV]}(I)=\codim_{A[\cU]}(J)$.
\end{proposition}

\begin{proof}
We note that the result is classical if $\cU$ is finite (as can be seen, for example, using Hilbert polynomials). We will use this twice in the proof of the general case.

First suppose that $\cV$ is finite and $I=\fp$ is prime. Note then that $J=\fq$ is prime as well. If $\fp_0 \subsetneq \cdots \subsetneq \fp_c=\fp$ is a chain of primes in $A[\cV]$, then letting $\fq_i$ be the extension of $\fp_i$, we get a chain of primes $\fq_0 \subsetneq \cdots \subsetneq \fq_c=\fq$ in $A[\cU]$, and so $\codim_{A[\cU]}(\fq) \ge c$, which shows $\codim_{A[\cU]}(\fq) \ge \codim_{A[\cV]}(\fp)$. 

Next suppose that $\fq_0 \subsetneq \cdots \subsetneq \fq_c=\fq$ is a chain of primes in $A[\cU]$. For each $0<i\le c$, pick $f_i \in \fq_i \setminus \fq_{i-1}$. Let $\cV'$ be a finite subset of $\cU$ containing $\cV$ and such that each $f_i$ belongs to $A[\cV']$. Then $\fq_{\bullet} \cap A[\cV']$ is a strict chain of primes ideals in $A[\cV']$, and so we see that $\codim_{A[\cV']}(\fp A[\cV']) \ge c$ (note that the contraction of $\fq$ to $A[\cV']$ is equal to the extension of $\fp$ to $A[\cV']$). However, $\codim_{A[\cV']}(\fp A[\cV'])=\codim_{A[\cV]}(\fp)$ by classical theory. We thus see that $\codim_{A[\cV]}(\fp) \ge c$, and so $\codim_{A[\cV]}(\fp) \ge \codim_{A[\cU]}(\fq)$. In particular, we have equality and this case has been proven.

Next, suppose still that $\cV$ is finite, but let $I$ be an arbitrary ideal. If $\fp$ is a codimension $c$ prime of $A[\cV]$ containing $I$ then $\fp A[\cU]$ is a codimension $c$ prime of $A[\cU]$, by the previous paragraph, containing $J$. We thus see that $\codim_{A[\cU]}(J) \le \codim_{A[\cV]}(I)$. Next, suppose that $\fq$ is a codimension $c$ prime of $A[\cU]$ containing $J$. By Proposition~\ref{prop:fg}, there is a finite subset $\cV'$ of $\cU$ (which we can assume contains $\cV$) such that $\fq$ is the extension of an ideal (necessarily prime) $\fq'$ of $A[\cV']$. By the previous paragraph, $\fq'$ has codimension $c$ in $A[\cV']$. Since $\fq'$ clearly contains the extension of $I$ to $A[\cV']$, we see that $\codim_{A[\cV']}(I A[\cV']) \le c$. But $\codim_{A[\cV']}(I A[\cV'])=\codim_{A[\cV]}(I)$ by classical theory, and so $\codim_{A[\cV]}(I) \le c$. We thus see that $\codim_{A[\cV]}(I) \le \codim_{A[\cU]}(J)$.

Finally, we treat the case where $\cV$ is arbitrary. Since $I$ is finitely generated, there is a finite subset $\cV_0$ of $\cV$ such that $I$ is the extension of an ideal $I_0$ of $A[\cV_0]$. Thus
\begin{displaymath}
\codim_{A[\cV]}(I)=\codim_{A[\cV_0]}(I_0)=\codim_{A[\cU]}(J),
\end{displaymath}
by two applications of the case where $\cV$ is finite.
\end{proof}

\begin{corollary} \label{cor:fincodim}
Let $A$ be a finitely generated $\bk$-algebra. Every finitely generated ideal of $A[\cU]$ has finite codimension.
\end{corollary}

\begin{proof}
Let $J$ be a finitely generated ideal of $A[\cU]$. Then $J$ is the extension of an ideal $I$ of some $A[\cV]$ with $\cV \subset \cU$ finite. Since $\codim_{A[\cV]}(I)\leq \dim A[\cV]<\infty$, Proposition~\ref{prop:codimex} implies that $\codim_{A[\cU]}(J)$ is finite as well.
\end{proof}

\begin{corollary} \label{cor:rs}
Let $\cU$ be a set, and let $f_1, \ldots, f_r \in \bk[\cU]$. Then $f_1, \ldots, f_r$ form a regular sequence if and only if the ideal $(f_1, \ldots, f_r)$ has codimension $r$.
\end{corollary}

\begin{proof}
Let $\cV$ be a finite subset of $\cU$ such that $f_1, \ldots, f_r \in \bk[\cV]$. Let $I$ (resp.\ $J$) be the ideal of $\bk[\cV]$ (resp.\ $\bk[\cU]$) generated by the $f_i$. The $f_i$ form a regular sequence in $\bk[\cV]$ (or in $\bk[\cU]$) if and only if the Koszul complex on the $f_i$ is exact; however, since $\bk[\cU]\subseteq \bk[\cV]$ is faithfully flat, the Koszul complex on the $f_i$ is exact over $\bk[\cU]$ if and only if it exact over $\bk[\cV]$.
\end{proof}

\begin{corollary} \label{cor:krull}
Let $A$ be a finitely generated $\bk$-algebra, let $\cU$ be a set, and let $J$ be a finitely generated ideal of $A[\cU]$ containing a nonzerodivisor $f$. Let $\ol{J}$ be the image of $J$ in $A[\cU]/(f)$. Then $\codim_{A[\cU]/(f)}(\ol{J})=\codim_{A[\cU]}(J)-1$.
\end{corollary}

\begin{proof}
Let $\cV$ be a sufficiently large finite subset of $\cU$ such that $A[\cV]$ contains $f$ and some finite generating set of $J$; thus $J$ is the extension of some ideal $I$ of $A[\cV]$. Let $B=A[\cV]$, which is a finitely generated $\bk$-algebra, and note that $A[\cU] = B[\cU']$, where $\cU'=\cU \setminus \cV$. Let $\ol{I}$ be the image of $I$ in $\ol{B}=B/(f)$. Since $\ol{J}$ is the extension of $\ol{I}$ to $\ol{B}[\cU']=A[\cU]/(f)$, we obtain $\codim_{\ol{B}}(\ol{I}) = \codim_{A[\cU]/(f)}(\ol{J})$ and $\codim_B(I) = \codim_{A[\cU]}(J)$ by two applications of Proposition~\ref{prop:codimex}. Finally since $\codim_{\ol{B}}(\ol{I})=\codim_B(I)-1$ by classical theory (for example, the principal ideal theorem and ~\cite[Corollary~13.4]{eisenbud}), the result follows.
\end{proof}

\begin{proposition} \label{prop:fgc}
Let $A$ be a finitely generated $\bk$-algebra and let $\cU$ be a set. Let $J$ be a finitely generated ideal of $A[\cU][y]$. Then $A[\cU] \cap J$ is also a finitely generated ideal.
\end{proposition}

\begin{proof}
Let $\cV$ be a finite subset of $\cU$ such that $J$ is the extension of an ideal $I$ of $A[\cV][y]$. Then $A[\cV] \cap I$ is finitely generated since $A[\cV]$ is noetherian. One easily sees that $A[\cU] \cap J$ is the extension of $A[\cV] \cap I$, which proves the result.
\end{proof}

\begin{corollary} \label{cor:contract}
Let $A$ be a finitely generated $\bk$-algebra, let $\cU$ be a set, let $R=A[\cU]$, and let $S=R[y]$. Let $I$ be a finitely generated ideal of $S$. Suppose that $I$ contains a positive degree monic polynomial, that is, an element of the form $y^n + \sum_{i=0}^{n-1} a_i y^i$ with $a_i \in R$ and $n>0$. Then $\codim_R(R \cap I) = \codim_S(I)-1$.
\end{corollary}

\begin{proof}
Let $f \in I$ be a monic polynomial. Let $\ol{I}$ be the image of $I$ in $S/(f)$. Then $R \to S/(f)$ is a finite flat extension of rings and $R \cap I$ is the contraction of $\ol{I}$ along this map. We thus see that $\codim_R(R \cap I)=\codim_{S/(f)}(\ol{I})$ by Proposition~\ref{prop:finflat}. But $\codim_{S/(f)}(\ol{I})=\codim_S(I)-1$ by Corollary~\ref{cor:krull}.
\end{proof}

\section{Stillman's conjecture via the ultraproduct ring} \label{s:ultra}

\subsection{Background on ultraproducts} \label{subsec:ultra}

For more details and references on ultraproducts, see \cite[\S 2.1]{schoutens}. Let $\cI$ be an infinite set. We fix a non-principal ultrafilter $\cF$ on $\cI$, which is a collection of subsets of $\cI$ satisfying the following properties:
\begin{enumerate}
\item $\cF$ contains no finite sets,
\item if $A \in \cF$ and $B \in \cF$, then $A \cap B \in \cF$,
\item if $A \in \cF$ and $A \subseteq B$, then $B \in \cF$,
\item for all $A \subseteq \cI$, either $A \in \cF$ or $\cI \setminus A \in \cF$ (but not both).
\end{enumerate}
We think of the sets in $\cF$ as neighborhoods of some hypothetical (and non-existent) point $\ast$ of $\cI$, and refer to them as such.  We say that some condition holds near $\ast$ if it holds in some neighborhood of $\ast$.

Given a family of sets $\{X_i\}_{i \in \cI}$, their ultraproduct is the quotient of the usual product $\prod_{i \in \cI} X_i$ in which two sequences $(x_i)$ and $(y_i)$ are identified if the equality $x_i=y_i$ holds near $\ast$. If $x$ is an element of the ultraproduct, we will write $x_i$ for the $i$th coordinate of $x$, keeping in mind that this is only well-defined in sufficiently small neighborhoods of $\ast$; in other words, we can think of $x$ as a germ of a function around $\ast$.

Suppose that each $X_i$ is a graded abelian group. We define the {\bf graded ultraproduct} of the $X_i$'s to be the subgroup of the usual ultraproduct consisting of elements $x$ such that $\deg(x_i)$ is bounded near $\ast$. 
The graded ultraproduct is a graded abelian group; in fact, it is the ultraproduct of the $X_i$'s in the category of graded abelian groups. The degree $d$ piece of the graded ultraproduct is the usual ultraproduct of the degree $d$ pieces of the $X_i$'s. We apply this construction in particular to the case where the $X_i$'s are graded rings; the graded ultraproduct is then again a graded ring.

\begin{example}
If $\bK$ is the ultraproduct of $\{\bk_i\}_{i\in \cI}$, then the graded ultraproduct of $\bk_i[x_1,\dots,x_n]$ (with standard grading) is $\bK[x_1,\dots,x_n]$ (also with standard grading).
\end{example}

In this subsection, we develop a few basic properties of graded ultraproduct rings.  We begin with a simple observation on adjoining variables to ultraproducts.

\begin{proposition} \label{prop:ultravar}
Let $\{R_i\}_{i \in \cI}$ be a family of graded rings with graded ultraproduct $\bS$. Let $y$ be a variable of degree $1$, and let $\wt{\bS}$ be the graded ultraproduct of the rings $R_i[y]$. Then the natural map $\bS[y] \to \wt{\bS}$ is an isomorphism.
\end{proposition}

\begin{proof}
Suppose that $f=\sum_{k=0}^d a_k y^k$ is an element of $\bS[y]$, and let $g$ be its image in $\wt{\bS}$. Then $g_i=\sum_{k=0}^d a_{k,i} y^k$. If $g=0$ then, passing to some neighborhood of $\ast$, we can assume $g_i=0$ for all $i$, which implies that $a_{k,i}=0$ for all $i$ and $k$, which implies that $a_k=0$ for all $k$, which shows that $f=0$. Thus the map is injective.

Next, suppose that $g$ is an element of $\wt{\bS}$ of degree $d$. Then we can write $g_i=\sum_{k=0}^d a_{k,i} y^k$ for each $i$, where the $a_{k,i}$'s are elements of $R_i$. Let $a_k$ be the element of $\bS$ defined by the sequence $(a_{k,i})$. Then $g$ is the image of $f=\sum_{k=0}^d a_k y^k$, and so the map is surjective.
\end{proof}

We next prove a simple result on base change:

\begin{proposition} \label{prop:basechange}
Let $\bk'/\bk$ be a finite field extension. Let $\{R_i\}_{i \in \cI}$ be a family of graded $\bk$-algebras with graded ultraproduct $\bS$. Let $R'_i=\bk' \otimes_{\bk} R_i$, and let $\bS'$ be the graded ultraproduct of $\{R'_i\}_{i \in \cI}$. Then the natural map $\bk' \otimes_{\bk} \bS \to \bS'$ is an isomorphism.
\end{proposition}

\begin{proof}
Let $\epsilon_1, \dots, \epsilon_d$ be a basis for $\bk'$ over $\bk$. We note that $R'_i$ is free as an $R_i$-module with basis $\epsilon_1 \otimes 1, \ldots, \epsilon_d \otimes 1$; similarly for $\bk' \otimes_{\bk} \bS$ over $\bS$. We claim that the $\epsilon$'s are also a basis for $\bS'$ over $\bS$.  Given $f = (f_i) \in \bS'$, we can decompose $f_i$ uniquely as $\epsilon_1 \otimes f_{i,1} + \epsilon_2 \otimes f_{i,2} +\cdots + \epsilon_d \otimes f_{i,d}$ where $f_{i,j}\in R_i$ for all $i,j$. We define $g_j = (f_{i,j})\in \bS$ for $1\leq j \leq d$, and we have the unique decomposition $f = g_1\epsilon_1+\cdots +g_d\epsilon_d$ in $\bS'$.
\end{proof}

We now examine how ideals in an ultraproduct relate to ideals in the original rings. Given a family of graded rings $\{R_i\}_{i \in \cI}$ and a family of ideals $\{I_i\}_{i \in \cI}$, we say that the $I_i$ are {\bf uniformly finitely generated} if there exists an integer $n$ such that $I_i$ is generated by at most $n$ elements for all $i$ in some neighborhood of $*$. The graded ultraproduct of $\{I_i\}_{i \in \cI}$ is the subset of elements $(r_i)_{i \in \cI}$ of the graded ultraproduct of $\{R_i\}_{i \in \cI}$ such that $r_i \in I_i$ for all $i$ in some neighborhood of $*$. It is an ideal of the graded ultraproduct of $\{R_i\}_{i \in \cI}$. From now on, we will generally drop the subscript $i \in \cI$.

\begin{proposition} \label{prop:ultraideal}
Let $\{R_i\}$ be a family of graded rings with graded ultraproduct $\bS$.
\begin{enumerate}[\indent \rm (a)]
\item Suppose that $\{I_i\}$ is a uniformly finitely generated family of homogeneous ideals. Then their graded ultraproduct $I$ is a finitely generated ideal of $\bS$.
\item Suppose that $\{I_i\}$ and $\{J_i\}$ are two uniformly finitely generated families of homogeneous ideals whose graded ultraproducts are equal. Then $I_i=J_i$ for all $i$ in some neighborhood of $*$.
\item Suppose that $I$ is a finitely generated homogeneous ideal of $\bS$. Then there exists a uniformly finitely generated family of homogeneous ideals $\{I_i\}$ with ultraproduct $I$.
\end{enumerate}
\end{proposition}

\begin{proof}
(a) Suppose that each $I_i$ is generated by $\le n$ elements; pick homogeneous generators $f_{1,i}, \ldots, f_{n,i}$ of each $I_i$. Let $f_1, \ldots, f_n$ be the elements of $\bS$ defined by these sequences. We claim that $I$ is generated by $f_1, \ldots, f_n$. Indeed, suppose that $g$ is a homogeneous element of $I$; thus, passing to a small enough neighborhood of $*$, we see that each $g_i$ is an element of $I_i$, and can thus be written as $\sum_{k=1}^n a_{k,i} f_{k,i}$ for some homogeneous elements $a_{k,i} \in R_i$. Let $a_k$ be the element of $\bS$ defined by the sequence $a_{k,i}$. Then $g=\sum_{k=1}^n a_k f_k$, proving the claim. (Note that for $k$ fixed, each $a_{k,i}$ is homogeneous of some degree, but that the degree may depend on $i$. However, the degree is bounded by the degree of $g$, and so in any small enough neighborhood of $*$, the degree of $a_{k,i}$ will be independent of $i$.)

(b) Suppose that $I_i$ and $J_i$ are each generated by at most $n$ elements for all $i$, and pick generators $f_{1,i}, \ldots, f_{n,i}$ and $g_{1,i}, \ldots, g_{n,i}$. Let $f_1, \ldots, f_n$ and $g_1, \ldots, g_n$ be the elements of $\bS$ these sequences define. By (a), the $f_k$'s and $g_k$'s generate the same ideal of $\bS$. Thus we have an expression $g_k=\sum_{j=1}^n a_j f_j$ for some $a_j \in \bS$, and so $g_{k,i} = \sum_{j=1}^n a_{j,i} f_{j,i}$ holds for all $i$ in some neighboorhood of $\ast$, and so $g_{k,i}$ belongs to the ideal $I_i$ for all such $i$. Since there are only finitely many $f$'s and $g$'s, we can pass to some common neighboorhood of $\ast$ so that $g_{k,i} \in I_i$ and $f_{k,i} \in J_i$ for all $i$ and $k$, and so $I_i=J_i$.

(c) Let $I$ be generated by $f_1, \ldots, f_n$. Let $f_k$ be represented by some sequence $(f_{k,i})$, and let $I_i$ be the ideal of $R_i$ generated by $f_{1,i}, \ldots, f_{n,i}$. Then the argument in (a) shows that $I$ is the ultraproduct of the $I_i$'s.
\end{proof}

Due to this proposition, we can unambiguously speak of the germ of a finitely generated homogeneous ideal $I$ of $\bS$. We denote these ideals by $I_i$, keeping in mind that they are only well-defined for $i$ sufficiently close to $\ast$. We next show that this construction interacts well with contraction.

\begin{proposition} \label{prop:ultracontract}
Let $\{R_i\}$ be a family of graded rings with graded ultraproduct $\bS$, and let $\{R'_i\}$ be a family of graded subrings of $\{R_i\}$ with graded ultraproduct $\bS'$. Let $I$ be a finitely generated homogeneous ideal of $\bS$, and suppose that $\bS' \cap I$ is a finitely generated ideal of $\bS'$. Then $(\bS' \cap I)_i = R'_i \cap I_i$ in a neighborhood of $\ast$.
\end{proposition}

\begin{proof}
Let $g_1, \ldots, g_m$ be generators for $\bS' \cap I$. Then $g_{k,i}$ belongs to $R'_i \cap I_i$ (in some neighborhood of $\ast$), and so $(\bS' \cap I)_i$ is contained in $R'_i \cap I_i$ (in some neighborhood of $\ast$), since the former is generated by $g_{1,i}, \ldots, g_{m,i}$. We now claim that the inclusion $(\bS' \cap I)_i \subset R'_i \cap I_i$ is an equality in some neighborhood of $\ast$. Assume not. Then we can find a sequence $(h_i)$ such that $h_i \in R'_i \cap I_i$ for all $i$, but in any neighborhood of $\ast$ there exists $i$ such that $h_i \not\in (\bS' \cap I)_i$. Let $h \in \bS$ be the element defined by $(h_i)$. Then $h \in \bS'$, since $h_i \in R'_i$ for all $i$, and $h \in I$, since $h_i \in I_i$ for all $i$. Thus $h \in \bS' \cap I$, and so $h=\sum_{k=1}^m a_k g_k$ for some $a_k \in \bS'$. But then $h_i=\sum_{k=1}^m a_{k,i} g_{k,i}$ holds in some neighborhood of $\ast$, which shows that $h_i \in (\bS' \cap I)_i$ in some neighborhood of $\ast$, a contradiction.
\end{proof}

We close this subsection with a result on strength in ultraproducts:

\begin{proposition} \label{prop:str}
Let $\{R_i\}$ be a family of graded rings with graded ultraproduct $\bS$. Suppose that the degree~$0$ piece of $R_i$ is a field $\bk_i$, so that the degree~$0$ piece of $\bS$ is the ultraproduct $\bK$ of these fields. Choose homogeneous elements $f_1, \ldots, f_r \in \bS$. Suppose that the collective strength of $f_{1,i}, \ldots, f_{r,i}$ is unbounded in all sufficiently small neighborhoods of $\ast$. Then $f_1, \ldots, f_r$ have infinite collective strength.
\end{proposition}

\begin{proof}
Suppose we have a relation $\sum_{j=1}^r a_j f_j = \sum_{k=1}^s g_k h_k$ where $a_i \in \bK$ are not all zero and $g_k$ and $h_k$ are elements of positive degree. Represent everything by sequences: $a_j=(a_{j,i})$, $g_k=(g_{k,i})$, and $h_k=(h_{k,i})$. Then, by definition of the ultraproduct, in all sufficiently small neighborhoods of $*$, we have $\sum_{j=1}^r a_{j,i} f_{j,i} = \sum_{k=1}^s g_{k,i} h_{k,i}$ and $a_{j,i} \ne 0$ for at least one $j$. But this shows that $f_{1,i},\ldots,f_{r,i}$ have collective strength $<s$ in this neighborhood of $\ast$.
\end{proof}

\subsection{The main theorems on ultraproduct rings} \label{ss:ultramain}

Let $\{\bk_i\}_{i \in \cI}$ be a family of perfect fields with ultraproduct $\bK$.  The field $\bK$ is also perfect, as if $\bK$ has characteristic $p>0$, then $\bk_i$ is perfect of characteristic of $p$ for all $i$ sufficiently close to $\ast$, and so one can take $p$th roots in $\bK$.  Let $R_i=\bk_i[x_1,x_2,\ldots]$ with standard grading, and let $\bS$ be the graded ultraproduct of the family $\{R_i\}_{i \in \cI}$.

\begin{theorem} \label{thm:S-poly-ring}
The ring $\bS$ is (isomorphic to) a polynomial ring.
\end{theorem}

\begin{proof}
We use the criteria of \S \ref{s:poly}. First suppose that $\bK$ has characteristic~0, and let us prove that $\bS$ has enough derivations (Definition~\ref{defn:enough1}). Let $f \in \bS$ be a non-zero homogeneous element of degree $d>0$. Passing to a neighborhood of $\ast$, we can assume that each $\bk_i$ has characteristic~0 or characteristic~$p$ with $p>d$, and that $f_i \ne 0$. For each $i$, let $a(i)$ be an index such that $x_{a(i)}$ appears in some monomial in $f_i$, and let $\partial_i$ be the derivation $\frac{\partial}{ \partial x_{a(i)}}$ of $R_i$. The derivations $(\partial_i)$ define a derivation $\partial$ on $\bS$. Since $\partial_i(f_i) \ne 0$ near $*$, we see that $\partial(f) \ne 0$, and so $\bS$ has enough derivations. Thus $\bS$ is a polynomial ring (Theorem~\ref{thm:poly}).

Now suppose that $\bK$ has characteristic~$p$, and let us prove that $\bS$ has enough Hasse derivations (Definition~\ref{defn:enough2}). Let $f \in \bS$ be a homogeneous element of positive degree that is not a $p$th power. Passing to a neighborhood of $\ast$, we can assume that each $f_i$ is not a $p$th power. For each $i$, let $a(i)$ be an index such that $x_{a(i)}$ appears in some monomial in $f_i$ with exponent not divisible by $p$, and let $\partial_i$ be the Hasse derivative on $R_i$ with respect to $x_{a(i)}$ (Example~\ref{ex:hasse}).  The Hasse derivations $\partial_i$ on the $R_i$ induce a Hasse derivation $\partial$ on $\bS$. Since $\partial_i(f_i) \ne 0$ near $*$, we see that $\partial(f) \ne 0$, and so $\bS$ has enough Hasse derivations. Thus $\bS$ is a polynomial ring (Theorem~\ref{thm:poly2}).
\end{proof}

\begin{theorem} \label{thm:codim in nhood}
If $I \subset \bS$ is a finitely generated homogeneous ideal, then $\codim_{\bS}(I) = \codim_{R_i}(I_i)$ for all $i$ sufficiently close to $\ast$.
\end{theorem}

\begin{proof}
Let $c=\codim_{\bS}(I)$, which is finite by Corollary~\ref{cor:fincodim}. If $c=0$ then $I=0$, and so $I_i=0$ for all $i$ sufficiently close to $\ast$, and so the formula holds. We now proceed by induction on $c$. 

Suppose the result holds for $c-1$, and let $I$ be an ideal of $\bS$ of codimension $c>0$. Let $f \in I$ be a non-zero homogeneous element. We would like for $\#\bk_i>\deg f$ to hold sufficiently close to $\ast$. Suppose this is not the case. Then, since the size of the $\bk_i$ is bounded near $\ast$, there must exist a single $q$ such that $\bk_i=\bF_q$ for all $i$ sufficiently close to $\ast$. It follows that $\bK=\bF_q$. Indeed, an element of $\bK$ is a sequence $(x_i)_{i \in \cI}$, and as each $x_i$ can only take finitely many values the sequence must be constant in a neighborhood of $\ast$. Let $e>0$ be an integer so that $q^e>\deg{f}$, let $\bk'_i=\bF_{q^e}$, let $R'_i=\bk'_i[x_1,x_2,\ldots]$, and let $\bS'$ be the ultraproduct of the $R'_i$. By Proposition~\ref{prop:basechange}, the natural map $\bF_{q^e} \otimes_{\bF_q} \bS \to \bS'$ is an isomorphism. Write $I'$ and $I'_i$ for the extension of the ideals $I$ and $I_i$ to $\bS'$ and $R_i'$, respectively. We have that $\codim_{R_i}(I_i) = \codim_{R_i'}(I_i')$ and $\codim_{\bS'}(I')=\codim_{\bS}(I)$. It thus suffices to prove that $\codim_{\bS'}(I')=\codim_{R_i'}(I_i')$ for all $i$ sufficiently close to $\ast$. Relabeling, we have reduced to the case where $\#\bk_i>\deg f$ holds in a neighborhood of $\ast$.

For each $i$, let $\gamma_i$ be an automorphism of $R_i$ such that $\gamma_i(f_i)$ is monic in $x_1$, at least for $i$ sufficiently close to $\ast$ (see Lemma~\ref{lem:monic}). The family $\{\gamma_i\}$ induces an automorphism $\gamma$ of $\bS$. Since codimension is invariant under automorphisms, we may replace $I$ with $\gamma(I)$, and so we can assume that $f_i$ is monic in $x_1$ for all $i$ sufficiently close to $\ast$.

Let $R'_i=\bk_i[x_2,\ldots]$ and let $\bS'$ be the ultraproduct of $\{R'_i\}$. We have $R_i=R'_i[x_1]$ for each $i$, and so $\bS \cong \bS'[x_1]$ by Proposition~\ref{prop:ultravar}. Under this identification, $f$ corresponds to a monic polynomial in $\bS'[x_1]$. Let $I'$ be the contraction of $I$ to $\bS'$, which is finitely generated by Proposition~\ref{prop:fgc}. We note that $I'_i$ is the contraction of $I_i$ to $R'_i$, for all $i$ sufficiently close to $\ast$, by Proposition~\ref{prop:ultracontract}. Corollary~\ref{cor:contract} implies that $\codim_{\bS'}(I')=\codim_{\bS}(I)-1=c-1$. Thus, by the inductive hypothesis, we have $\codim_{R'_i}(I'_i)=c-1$ for all $i$ sufficiently close to $\ast$. By Corollary~\ref{cor:contract} again, $\codim_{R_i}(I_i)=\codim_{R'_i}(I'_i)+1=c$. The result follows.
\end{proof}

\begin{lemma} \label{lem:monic}
Let $\bk$ be a field, let $R=\bk[x_1,x_2,\ldots]$, and let $f \in R$ be a non-zero homogeneous element. If $\#\bk>\deg f$ then there exists an automorphism $\gamma$ of $R$ (as a graded $\bk$-algebra) such that $\gamma(f)$ is monic in $x_1$.
\end{lemma}

\begin{proof}
(Compare with \cite[Lemma~13.2(c)]{eisenbud}.) We may assume that $f$ lies in $\bk[x_1,\dots,x_n]$ for some $n$. Let $d=\deg f$.  We consider an automorphism of the form $\gamma(x_i)=x_i$ for $i=1$ or $i>n$ and $\gamma(x_i)=x_i-a_ix_1$ for $2\leq i \leq n$, where $a_i\in \bk$.  The coefficient of $x_1^d$ in $\gamma(f)$ can be viewed as an inhomogeneous polynomial $g(a_2,\dots,a_n)$, with $\deg(g)\leq d$.  Thus, as long as $\#\bk>d$, we can find some choice of $a_2,\dots,a_n$ where $g(a_2,\dots,a_n)\ne 0$,
\end{proof}

\begin{corollary}\label{cor:reg seq ultra}
Let $f_1, \ldots, f_r$ be homogeneous elements of $\bS$. Then $f_1, \ldots, f_r$ form a regular sequence in $\bS$ if and only if $f_{1,i}, \ldots, f_{r,i}$ form a regular sequence in $R_i$ for all $i$ sufficiently close to $\ast$.
\end{corollary}

\begin{proof}
This follows from Theorem~\ref{thm:codim in nhood} and Corollary~\ref{cor:rs}. 
\end{proof}

\subsection{Stillman's conjecture}\label{ss:stillman}

\begin{theorem} \label{thm:reg}
Given positive integers $d_1, \ldots, d_r$ there exists an integer $N=N(d_1,\ldots,d_r)$ with the following property. If $f_1, \ldots, f_r$ are homogeneous elements of $\bk[x_1, \ldots, x_n]$, for any perfect field $\bk$ and any $n$, of degrees $d_1, \ldots, d_r$ and collective strength at least $N$ then $f_1, \ldots, f_r$ form a regular sequence.
\end{theorem}

\begin{proof}
Suppose that the theorem is false. Then for each $i\in \bN$, we can find $f_{1,i}, \dots,f_{r,i} \in \bk_j[x_1,x_2,\dots]$, with $\bk_i$ perfect, which fail to form a regular sequence and where the collective strength goes to $\infty$ as $i\to \infty$.  Choosing $\cI=\bN$, we let $f_1=(f_{1,i}), \dots, f_r=(f_{r,i})$ be the corresponding collection in $\bS$.  By Proposition~\ref{prop:str}, $f_1,\dots,f_r$ have infinite collective strength.  However, by Corollary~\ref{cor:reg seq ultra}, $f_1,\dots,f_r$ fail to form a regular sequence.  This contradicts Theorem~\ref{polycrit}.
\end{proof}

For completeness, we now illustrate how Theorem~\ref{thm:reg} implies the existence of small subalgebras and Stillman's conjecture. This implication is 
essentially the same as in \cite{ananyan-hochster}.

\begin{theorem} \label{thm:smallalg}
Given positive integers $d_1, \ldots, d_r$ there exists an integer $s=s(d_1,\ldots,d_r)$ with the following property. If $f_1, \ldots, f_r$ are homogeneous elements of $\bk[x_1, \ldots, x_n]$, for any perfect field $\bk$ and any $n$, with $\deg(f_i)=d_i$, then:
\begin{enumerate}[\rm \indent (a)]
	\item  There exists a regular sequence $g_1, \ldots, g_s$ in $\bk[x_1, \ldots, x_n]$, where each $g_i$ is homogeneous of degree at most $\max(d_1, \ldots, d_r)$, such that $f_1, \ldots, f_r$ are contained in the subalgebra $\bk[g_1, \ldots, g_s]$.
	\item  The ideal $( f_1,\ldots,f_r) $ has projective dimension at most $s$.
\end{enumerate}
\end{theorem}

\begin{proof}
(a)  To each sequence $\bd=(d_1,\dots,d_r)$ we attach a monomial $y(\bd)=y_1^{b_1}y_2^{b_2}\cdots$ where $b_j$ is the number of times $j$ appears in $\bd$.  If there is an ideal $(f_1,\dots,f_r)$ of type $\bd$ that fails to be a regular sequence, then by Theorem~\ref{thm:reg} there is some $N$, depending only on $\bd$, such that some $\bk$-linear homogeneous combination of the $f_i$ has strength $\leq N$. Without loss of generality, we may replace one of our elements with this linear combination, and call it $f_i$. Taking $f_i=\sum_{j=1}^N a_jg_j$, and replacing $f_i$ by the $a_j$ and the $g_j$, we get an ideal of type $\bd'$ and where $y(\bd)<y(\bd')$ in the lexicographic order (but where the variables are checked in reverse order), and where the difference in total degree is at most $2N-1$.  In particular, given $y(\bd)$ there are only a finite number of possible monomials $y(\bd')$ that could arise in this way.  The descendants of $y(\bd)$ thus form a tree with finitely many branches at each node and with no infinite chains, and there are thus only finitely many descendants of $y(\bd)$.  Letting $s$ be the max total degree of a descendant of $y(\bd)$, it follows that $f_1,\dots,f_r$ can be embedded in a subalgebra $\bk[g_1,\dots,g_s]$ where the $g_i$ form a regular sequence.

(b) Choose $g_1,\dots,g_s$ as in (a). Since $g_1,\dots,g_s$ form a regular sequence, the extension $\bk[g_1,\dots,g_s]\subseteq \bk[x_1,\dots,x_n]$ is flat.  Thus, if $G$ is the minimal free resolution of $(f_1,\dots,f_r)$ over $\bk[g_1,\dots,g_s]$, then the extension of $G$ to $\bk[x_1,\dots,x_n]$ is the minimal free resolution of this ideal over $\bk[x_1,\dots,x_n]$.  In particular, the projective dimension of $(f_1,\dots,f_r)$ is $\leq s$.  
\end{proof}

\section{The inverse limit ring} \label{s:limit}

\subsection{Inverse limit polynomial ring}\label{subsec:inv}

Recall that $A\vars$ denotes the inverse limit of the standard-graded polynomial rings $A[x_1,\ldots,x_n]$ in the category of graded rings. We let $K$ denote a ring containing $A$, and we write $\alpha_n \colon K \otimes_A A\vars \to K[x_1,\dots,x_n]$ for the natural surjection. We set $\bR = K \otimes_A A\vars$.

The following hypothesis will be used repeatedly.

\begin{hypothesis}\label{defn:bR}
$A$ is an integral domain with fraction field $K$. If the characteristic $p$ of $K$ is positive, we assume furthermore that the Frobenius map on $A$ is surjective (so that $K$ is perfect) and that $K$ is infinite.
\end{hypothesis}

\begin{remark}\label{rmk:replace}
If $A$ is normal and its fraction field $K$ is perfect, then because $a^{1/p}$ satisfies the integral equation $x^p-a$, it lies in $A$.
Thus, we can often arrange to satisfy Hypothesis~\ref{defn:bR} by replacing $A$ with its integral closure in an algebraic closure of $K$.
\end{remark}

The following is an analogue of Theorem~\ref{thm:S-poly-ring} and Corollary~\ref{cor:reg seq ultra}.  It implies Theorem~\ref{introthm1}.

\begin{theorem}\label{thm:inverse is poly}
  Suppose Hypothesis~\ref{defn:bR} holds (except we do not require $K$ to be infinite). Then $\bR$ is a polynomial ring.
\end{theorem}

\begin{proof}
We use the criteria of \S \ref{s:poly}.  If $p=0$, then the partial derivatives $\frac{d}{dx_i}$ show that $\bR$ has enough derivations.

Now suppose that $p>0$. We claim that the Hasse derivatives corresponding to $\frac{d}{dx_i}$ (Example~\ref{ex:hasse}) provide $\bR$ with enough Hasse derivations. Let $f \in \bR$ be such that $\frac{d}{dx_i} f = 0$ for all $i$. This implies that $f \in K\otimes_A A\invl x_1^p,x_2^p,\dots \invr$. In particular, we can write $f = g/a$ where $a \in A$ and $g \in A\invl x_1^p, x_2^p, \dots \invr$. Since the Frobenius map is surjective on $A$ and $K$, both $g$ and $a$ are $p$th powers, which implies that $f$ is also a $p$th power.
\end{proof}

\begin{remark} \label{rmk:relax2}
The perfectness hypothesis in Theorem~\ref{thm:inverse is poly} can be relaxed. For example, suppose $\bk$ is a field of characteristic~$p$ such that $\bk$ is a finite extension of the subfield $\bk^p$, and let $R=\bk \vars$. Then $\bk R^p$ consists exactly of all (possibly infinite) $\bk$-linear combinations of $p$th powers of monomials; this uses the hypothesis on $\bk$. Thus if $f \not\in \bk R^p$ then some Hasse derivative will not kill $f$, and so $R$ has enough derivations, and is thus a polynomial ring by Remark~\ref{rmk:relax}.
\end{remark}

\begin{theorem} \label{thm:codim in truncations}
Suppose Hypothesis~\ref{defn:bR} holds.  Let  $f_1,\dots,f_s\in \bR$ and let $I=(f_1,\dots,f_s)$.  
\begin{enumerate}[\rm \indent (a)]
	\item  For any $n\gg 0$, we have that $\codim_\bR(I) = \codim_{K[x_1,\dots,x_n]}(\alpha_n(I))$.
	\item  The sequence $f_1,\dots,f_s$ forms a regular sequence if and only if $\alpha_n(f_1),\dots,\alpha_n(f_s)$ forms a regular sequence for all $n\gg 0$.
	\item  If $\alpha_n(f_1),\dots,\alpha_n(f_s)$ forms a regular sequence for some $n$, then $\alpha_m(f_1),\dots,\alpha_m(f_s)$ forms a regular sequence for all $m\geq n$.
\end{enumerate}
\end{theorem}

\begin{proof}
(a) We prove this by induction on $c=\codim_\bR(I)$, which is finite by Corollary~\ref{cor:fincodim}. If $c=0$ then $I=0$ and the statement is immediate.  Now let $c>0$ and pick $f \in I$ nonzero.  Let $n$ large enough so that $\alpha_n(f)\ne 0$.  Since $K$ is infinite, there is a graded $K$-algebra automorphism $\gamma$ of $K[x_1,\dots,x_n]$ such  that $\gamma \alpha_n(f)$ is monic over $K[x_2,\dots,x_n]$ (see \cite[Lemma~13.2(c)]{eisenbud} and its proof).  If $\gamma'$ is the automorphism of $\bR$ which acts by $\gamma$ on $x_1,\dots,x_n$ and which acts trivially on the other $x_i$, then $\alpha_n(\gamma' f)=\gamma \alpha_n f$.  We may thus assume that $f$ is monic over $K\otimes_A A\invl x_2,x_3,\dots \invr$. 
The rest of the proof is essentially identical to the proof of Theorem~\ref{thm:codim in nhood}.

(b) This is an immediate consequence of (a) and Corollary~\ref{cor:rs}.

(c)  By Corollary~\ref{cor:rs}, we have $\codim \alpha_n(I)=s$ and it suffices to show that $\codim \alpha_{n+1}(I)=s$.  Since $K[x_1,\dots,x_{n+1}]/(\alpha_{n+1}(I)+(x_{n+1}))$ is isomorphic to $K[x_1,\dots,x_n]/\alpha_n(I)$, the principal ideal theorem implies that $\codim \alpha_{n+1}(I)$ is either $s$ or $s+1$.
However, $\alpha_{n+1}(I)$ is generated by $s$ elements, so its codimension is at most $s$.  Thus $\codim \alpha_{n+1}(I)=s$.
\end{proof}

\begin{definition}\label{defn:restriction to point}
Fix a ring $A$, a field $\bk$, and a point $y\in\Spec(A)(\bk)$. For $f\in A\vars$, we let $f_y$ denote the image of $f$ in $\bk\vars$.  Similarly, for an $A\vars$-module $M$, we let $M_y=\bk\vars\otimes_{A\vars} M$.

If instead $f\in A[x_1,\dots,x_n]$, then we let $f_y$ denote its image in $\bk[x_1,\dots,x_n]$.  If $M$ is an $A[x_1,\dots,x_n]$-module, then we let $M_y=\bk[x_1,\dots,x_n]\otimes_{A[x_1,\dots,x_n]}M$.
\end{definition}

\begin{remark}
We note that $\bk\otimes_A A\vars$ is not generally isomorphic to $\bk\vars$: $\bk \otimes_A A \vars$ consists of those infinite series in $x_1,x_2,\dots$ whose coefficients have a common denominator. For example, if $A = \bZ$ and $\bk=\bQ$, then $\sum_{i \ge 1} x_i/i$ is an element of $\bQ\vars$ but is not an element of $\bQ \otimes \bZ\vars$.
\end{remark}

\begin{corollary}\label{cor:zariski condition}
Suppose Hypothesis~\ref{defn:bR} holds. Let $f_1,\dots,f_s \in A\vars$ be elements whose images in $\bR$ form a regular sequence. There exists a dense open set $U\subseteq \Spec(A)$ such that for any algebraically closed field $\bk$ and any $y\in U(\bk)$, the elements $f_{1,y},\dots,f_{s,y}$ form a regular sequence in $\bk\vars$.
\end{corollary}

\begin{proof}
By Theorem~\ref{thm:codim in truncations}, there is some $n$ so that $\alpha_n(f_1),\dots,\alpha_n(f_s)\in K[x_1,\dots,x_n]$ forms a regular sequence.  Let $g_i=\alpha_n(f_i)$, considered as an element of $A[x_1,\dots,x_n]$.  Let $Q=A[x_1,\dots,x_n]/(g_1,\dots,g_s)$ and let $\pi\colon \Spec(Q)\to \Spec(A)$. Since the generic fiber of $\pi$ has dimension $n-s$, it follows that the locus $U\subseteq \Spec(A)$ of points whose fibers have dimension $n-s$ is dense and Zariski open by semicontinuity of fiber dimension \stacks{05F6}.

Let $\bk$ be an algebraically closed field and let $y\in U(\bk)$.  
Since $\dim(Q\otimes_A \bk)=n-s$,
it follows that $g_{1,y},\dots,g_{s,y}$ forms a regular sequence.  But $g_{i,y}$ equals $\alpha_n(f_{i,y})$, and thus by Theorem~\ref{thm:codim in truncations}(c) and (b), we have that $f_{1,y},\ldots,f_{s,y} $ forms a regular sequence.
\end{proof}

\begin{lemma}\label{lem:flat extension}
If $\bk$ is a perfect field and $f_1,\dots,f_s\in \bk\vars$ is a regular sequence, then $i' \colon \bk[f_1,\dots,f_s]\to \bk\vars$ is faithfully flat.
\end{lemma}

\begin{proof}
By Theorem~\ref{thm:inverse is poly}, we can write $\bk\vars \cong \bk[{\cV}]$. There is a finite subset ${\cH}\subseteq {\cV}$ such that ${f_1},\dots,{f_s}\in \bk[{\cH}]$.  Since the ${f_i}$ form a regular sequence, we can extend this to a maximal regular sequence, ${f_1},\dots,{f_s},g_1,\dots,g_{r}$ on $\bk[{\cH}]$.  The map $i'$ factors as
\[
\bk[{f_1},\dots,{f_s}]\overset{i_1}{\longrightarrow} \bk[{f_1},\dots,{f_s},g_1,\dots,g_{r}] \overset{i_2}{\longrightarrow} \bk[{\cH}] \overset{i_3}{\longrightarrow} \bk\vars.
\]
For each extension, the larger ring is free over the smaller ring.  Both $i_1$ and $i_3$ are extensions of polynomial rings.  For $i_2$, freeness follows from ~\cite[Proposition~2.2.11]{bruns-herzog} (the statement there is for a local ring, but the proof also works for a graded ring). 
\end{proof}

\subsection{Constant Betti tables over an open subset}\label{sec:inverse limit proof}
For a graded ring $R$ with $R_0=\bk$ a field, 
we set
 $\beta_{i,j}(M)=\dim_{\bk} \Tor^R_i(M,\bk)_j$.  The {\bf Betti table} of $M$ is the collection of all $\beta_{i,j}$.  

\begin{definition}\label{defn:constant betti}
Let $A$ be a commutative ring and let $U\subseteq \Spec(A)$ be a locally closed subset.  Let $M$ be a finitely presented, graded module over either $A\vars$ or over a polynomial ring over $A$.  We say that $M$ {\bf has a constant Betti table over $U$} if for every algebraically closed field $\bk$ and every $y\in U(\bk)$, the Betti table of $M_y$ is the same.
(Recall that $M_y$ is defined in Definition~\ref{defn:restriction to point}).
\end{definition}

The following lemma, which is likely known to experts, shows that a finitely presented module over a finite polynomial ring has a constant Betti table over an open subset.
\begin{lemma}\label{lem:finite betti}
Let $A$ be an integral domain and let $R=A[y_1,\dots,y_r]$ be a graded polynomial ring over $A$, with $\deg(y_i)\geq 1$ for $1\leq i \leq r$.  If $M'$ is a finitely presented, graded $R$-module, then $M'$ has a constant Betti table over some dense, open subset $U\subseteq \Spec(A)$.
\end{lemma}
\begin{proof}
Let $K$ be the fraction field of $A$ and let $G'_K=[0\to G'_{K,p} \overset{\partial_p}{\to}\cdots  \overset{\partial_1}{\to} G'_{K,0}]$ be the minimal free resolution of $K\otimes_A M'$ over $K[y_1,\dots,y_r]$. Represent each $\partial_i$ by a matrix $\phi_i$.  The entries of each $\phi_i$ have positive degree and, by multiplying by an element in $A$ if needed, we may assume that the entries of each $\phi_i$ also lie in $R$.  These matrices can then be used to define a bounded, graded complex $G'$ of free $R$-modules.  
By construction, $K\otimes_A \coker(G'_1\to G'_0)$ is isomorphic to $K\otimes_AM'$.  Since both $M'$ and $\coker(G'_1\to G'_0)$ are finitely presented, this isomorphism extends to an isomorphism over $A_g$ for some $g\in A$.

Let $B$ be the subring of $A$ generated by all of the coefficients of the polynomials that appear as entries in the differentials $\partial_i$ and let $S = B[y_1,\dots,y_r]$. Then we can also use the $\partial_i$ to define a complex $H$ over $S$ with the property that $H \otimes_B A \cong G'$.

Define $N$ be the direct sum of the homology modules $\rH_i(H)$ for $0\leq i \leq p$ along with the images of the differentials $\partial_i$ for $i=1,\dots,p$. Since $N$ is a finitely generated module over the finitely presented extension $B\to B[y_1,\dots,y_r]$,~\stacks{051S} implies that there is some $h\in B$ such that $N_h$ is free over $B_h$. Working over $B_h$ and $A_h$, we see that base extension of $H_h$ to $A_h$ and also $K$ commutes with taking homology (due to flatness of the images of $\partial_i$ over $B_h$). But since they are free over $B_h$, this implies that the homology of both $H_h$ and $G'_h$ vanish in positive degrees.

In sum, if $f=gh$, then $G'_f$  is a free resolution of $M'_f$, and $M'_f$ is a flat $A_f$-module. Let $U=\Spec(A_f)$.  Let $\bk$ be a field and $y\in U(\bk)$.  Since $M'_f$ is flat over $A_f$, $\bk\otimes_{A_f} G'_f$ is a free resolution of $M'_y$.  The resolution is minimal since each entry of $\phi_i$ has positive degree, and this remains true under localization at $f$ and specialization to $\bk$.  The Betti table of $M'_y$ is thus determined by the free modules in $G'_f$, and so it does not depend on $y$.
\end{proof}

\begin{lemma}\label{lem:first step}
Suppose Hypothesis~\ref{defn:bR} holds. Let $M$ be a finitely presented, graded $A\vars$-module.  There exist:
\begin{enumerate}[\rm \indent (a)]
	\item  elements $f_1,\dots,f_s\in A\vars$ whose images in $\bR$ form a regular sequence, and
	\item  an element $g\in A$ and a finitely presented $A_g[f_1,\dots,f_s]$-module $M'$ such that the extension of $M'$ to $A_g\otimes_A A\vars$ is isomorphic to $A_g\otimes_A M$.
\end{enumerate}
\end{lemma}

\begin{proof}
Let $\cU$ be a set of homogeneous elements of $\bR_+$ such that $\bR=K[\cU]$.  Since any $f\in \bR$ can be written as a fraction with numerator in $A\vars$ and denominator in $A$, we may rescale each element of $\cU$ so that it lies in $A\vars$.  Since $\bR=K[\cU]$ is a polynomial ring, for any element $f\in \bR$, there is a finite subset $\cU' \subseteq \cU$, and an element $\gamma \in A$ such that $f$ lies in $A_\gamma[\cU']$.  The same holds for any finite collection of elements in $\bR$.  

Let $\Phi$ be a finite presentation matrix of $M$. By the above discussion, we can find distinct elements $f_1,\dots,f_s\in \cU$ and $g\in A$ such that each entry of $\Phi$ lies in $A_g[f_1,\dots,f_s]$.  Let $\Phi'$ be the same matrix as $\Phi$, but considered as a map of graded, free $A_g[f_1,\dots,f_s]$-modules and let $M'$ be the cokernel of $\Phi'$.  By construction, the extension of $M'$ to $A_g\otimes_A A\vars$ is isomorphic to $A_g\otimes_A M$. The elements $f_1, \ldots, f_s \in \bR$ form a regular sequence as they are ``variables'' (elements of $\cU$).
\end{proof}

\begin{theorem}\label{thm:open subset}
Suppose Hypothesis~\ref{defn:bR} holds. If $M$ is a finitely presented, graded $A\vars$-module, then $M$ has a constant Betti table over some dense open subset $U\subseteq \Spec(A)$.
\end{theorem}
\begin{proof}
We apply Lemma~\ref{lem:first step}, and let $M'$ be the $A_g[f_1,\dots,f_s]$-module satisfying the conclusion of that lemma.
Applying Lemma~\ref{lem:finite betti} to $M'$, we can assume that $M'$ has constant Betti table over a dense open subset $U_1\subseteq \Spec(A_g)$.  By Corollary~\ref{cor:zariski condition}, we can find a dense open subset $U_2\subseteq \Spec(A_g)$ where for all algebraically closed fields $\bk$ and all  $y\in U_2(\bk)$, the sequence $f_{1,y}, \dots, f_{s,y}$ forms a regular sequence.  We let $U=U_1\cap U_2$.  

Let $\bk$ be an algebraically closed field and let $y\in U(\bk)$.  We have a commutative diagram
\[
\xymatrix{
&A_g[f_1,\dots,f_s]\ar[r]\ar[d]&\bk[f_{1,y},\dots,f_{s,y}]\ar[d]_{i'}\\
A\vars \ar[r]&A_g\otimes_A A\vars\ar[r] &\bk\vars
}
\]
where the extension of the $\bk[f_{1,y}, \dots, f_{s,y}]$-module $M'_y$ by $i'$ is $M_y$.  
By Lemma~\ref{lem:flat extension}, $i'$ is faithfully flat, and thus the Betti table of $M'_y$ is the same as the Betti table of $M_y$.  Since $M'$ has a constant Betti table over $U$, the module $M$ also has a constant Betti table over $U$.
\end{proof}

\begin{corollary}\label{cor:constant Betti}
(We do not assume Hypothesis~\ref{defn:bR}.)
Let $A$ be an integral domain.  If $M$ is a finitely presented, graded $A\vars$-module, then $M$ has a constant Betti table over some dense open subset $U\subseteq \Spec(A)$.
\end{corollary}

\begin{proof}
Let $\ol{A}$ be the integral closure of $A$ in an algebraic closure of $K$, and let $\ol{K}$ be the fraction field of $\ol{A}$.  
Let $\overline{M}$ be the extension of $M$ to $\overline{A}\vars$. Since $\overline{A}$ and $\overline{K}$ satisfy Hypothesis~\ref{defn:bR} (see Remark~\ref{rmk:replace}), Theorem~\ref{thm:open subset} implies that $\ol{M}$ has a constant Betti table over some dense open subset $\overline{U}\subseteq \Spec(\ol{A})$.

Since the integral morphism $\Spec(\ol{A})\to \Spec(A)$ is closed~\stacks{01WM}, the image of $\ol{U}$ in $\Spec(A)$ contains a dense open set $U$.
Let $\bk$ be an algebraically closed field and let $y\in U(\bk)$. By integrality, there is $\bk$-point $y'$ lying over $y$, and by definition of $U$, $y'\in \ol{U}(\bk)$.  The map $y'\to y$ induces an isomorphism of $\overline{M}_{y'}$ and $M_y$ as $\bk\vars$-modules, and they therefore have the same Betti table.  Thus $M$ has a constant Betti table over $U$.
\end{proof}

\begin{example}
  Consider the case when $M$ is a quotient of $A\vars$ by $r$ linear forms $f_i = \sum_j a_{i,j} x_j$. Generically, these forms are linearly independent. More precisely, this is true over the complement of the vanishing locus of the $r \times r$ minors of the matrix $(a_{i,j})$. This gives a dense open set where $M_y$ is resolved by a Koszul complex, and hence the Betti numbers are given by $\beta_{i,i} = \binom{r}{i}$ for $0 \le i \le r$ and $0$ otherwise.

  As a second example, consider when $M$ is defined by determinantal conditions. Let $f_i= \sum_j a_{i,j} x_j$ for $i=1,\dots,6$ and consider the matrix $\begin{bmatrix} f_1 & f_2 & f_3 \\ f_4 & f_5 & f_6   \end{bmatrix}$. Let $M$ be the quotient of $R$ by the $2 \times 2$ determinants of this matrix. General facts about determinantal loci tell us that this ideal has codimension $\le 2$. Having codimension 2 is an open condition, and in that case, the nonzero Betti numbers are given by $\beta_{0,0} = 1$, $\beta_{1,2} = 3$, $\beta_{2,3} = 2$ (for example, by the Hilbert--Burch theorem).
\end{example}

\subsection{Connection with $\GL$-noetherianity and Stillman's conjecture}

We now combine Corollary~\ref{cor:constant Betti} with~\cite{draisma} to prove Stillman's conjecture.  

Throughout this section we fix a ground field $\bk$.  Fix degrees $d_1,\dots,d_r$. Let $S$ be the set of pairs $(\alpha,i)$ where $1 \le i \le r$ and $\alpha$ ranges over all exponent vectors of degree $d_i$ in the variables $x_1,x_2,\dots$. Let $\bA = \bk[c_{\alpha,i} \mid (\alpha,i) \in S]$. For $1\leq i \leq r$, let $\widetilde{f}_i=\sum c_{\alpha,i}x^\alpha \in \bA\vars$ be a universal polynomial of degree $d_i$.   We let $Q=\bA\vars/(\widetilde{f}_1,\dots,\widetilde{f}_r)$.
If $\bk'$ is a field over $\bk$, then there is a bijection between $\Spec(\bA)(\bk')$ and tuples $f_1,\dots,f_r\in \bk'\vars$ where $\deg(f_i)=d_i$; with notation from Definition~\ref{defn:restriction to point}, this bijection is given by $y\in \Spec(\bA)(\bk') \leftrightarrow \widetilde{f}_{1,y},\dots,\widetilde{f}_{r,y}$. 

There is a natural change of basis action by the group scheme $\GL_\infty$ on $\Spec(\bA)$.  The $\bA\vars$-module $Q$ is equivariant with respect to this action.

\begin{theorem}\label{thm:finite Betti}
The space $\Spec(\bA)$ decomposes into a finite disjoint union of locally closed subsets $\{U_j\}$ such that $Q$ has a constant Betti table over $U_j$ for each $j$. In particular, there are only finitely many distinct Betti tables among all ideals $(f_1,\dots,f_r)\subseteq  \bk'[x_1,\dots,x_n]$ generated in degrees $d_1,\dots,d_r$, for all $n$ and all fields $\bk'$ over $\bk$.
\end{theorem}

\begin{proof}
Applying Corollary~\ref{cor:constant Betti}, we have that $Q$ has a constant Betti table over a dense, open subset $U'\subseteq \Spec(\bA)$.  Let $U$ be the union of all $\GL_\infty$ translates of $U'$. Since Betti tables are $\GL_\infty$-invariant, $Q$ has a constant Betti table over $U$.  By~\cite[Theorem 1]{draisma}, $\Spec(\bA)\setminus U$ consists of finitely many irreducible components, each of which is $\GL_\infty$-invariant.  Passing to a component, we can apply the same argument. Continuing in this way, we obtain the desired stratification of $\Spec(\bA)$, and it is finite by~\cite[Theorem 1]{draisma}.

For any field $\bk'$, the canonical map $\bk'[x_1,\ldots,x_n] \otimes_{\bk'} \bk' \invl x_{n+1}, \ldots \invr \to \bk'\vars$ is an isomorphism. It follows that for $f_1, \ldots, f_r \in \bk'[x_1,\ldots,x_n]$ the Betti table of the quotient ring $\bk'[x_1,\ldots,x_n]/(f_1, \ldots, f_r)$ is the same as that of $\bk' \vars/(f_1, \ldots, f_r)$, and this implies the final statement of the theorem for algebraically closed fields $\bk'$. To get the statement for arbitrary $\bk'$, we let $\ol{\bk'}$ be an algebraic closure of $\bk'$ and note that the extension $\bk'[x_1,\dots,x_n] \to \ol{\bk'}[x_1,\dots,x_n]$ is faithfully flat, and hence Betti tables are unchanged under this extension.
\end{proof}

\begin{remark}
It would be interesting to extend~\cite[Theorem 1]{draisma} to spaces over $\bZ$, as this would yield characteristic free bounds in the above result.
\end{remark}

\begin{remark}
Theorem~\ref{thm:finite Betti} slightly generalizes Stillman's conjecture, as it also applies to ideals $(f_1,\dots,f_r)$ in $\bk\vars$ that use an infinite number of variables.
\end{remark}

\begin{remark}
The proof of Theorem~\ref{thm:finite Betti} is much less self-contained than our ultraproduct proof, however it is distinctly different in character: it does not rely on the notion of strength, but rather on a generalized noetherianity principle. This is pursued in more detail in~\cite{erman-sam-snowden} to obtain generalizations of Stillman's conjecture.  
\end{remark}

\end{document}